\documentclass{amsart}
\usepackage{amsmath, palatino, mathpazo, amsfonts, amssymb, mathtools}
\usepackage[mathscr]{eucal}
\usepackage[all]{xy}
\usepackage{datetime}

\usepackage{amsthm}
\usepackage {tikz}
\usepackage {tgbonum}                 
\usepackage {amsmath}                   
\usepackage {amssymb}                   
\usepackage {bibnames}                  
\usepackage {cite}                      
\usepackage {url}                       
\usepackage {varioref}                  
\usepackage{comment}

\usepackage{tabularx}
\usepackage{array}

\title{Integrality of genus zero Gopakumar-Vafa type invariants of semi-positive varieties}

\newtheorem{theorem}{Theorem}[section]
\newtheorem{lemma}[theorem]{Lemma}
\newtheorem{proposition}[theorem]{Proposition}
\newtheorem{corollary}[theorem]{Corollary}

\newtheorem*{main theorem*}{Main Theorem}

\theoremstyle{remark}
\newtheorem{remark}[theorem]{Remark}

\newtheorem{definition}[theorem]{Definition}

\numberwithin{equation}{section}

\newcommand{\vdim}{\operatorname{vdim}}

\newcommand{\M}{\overline{\mathcal{M}}}
\newcommand{\K}{\mathcal{K}}

\newcommand{\vir}{{\rm vir}}
\newcommand{\ev}{{\rm ev}}
\newcommand{\Res}{{\rm Res}}
\newcommand{\td}{{\rm td}}
\newcommand{\ch}{{\rm ch}}

\newcommand{\Tr}{{\rm Tr}}

\newcommand{\fake}{{\rm fake}}

\newcommand{\KK}{{K}}
\newcommand{\Coeff}{{\rm Coeff}}
\newcommand{\leg}{{\rm leg}}
\newcommand{\arm}{{\rm arm}}
\newcommand{\tail}{{\rm tail}}

\def\O{\mathcal{O}}

\newcommand{\GW}{{\mathrm{GW}}}

\newcommand{\GV}{{\mathrm{GV}}}
\newcommand{\QK}{{\mathrm{QK}}}

\newcommand{\ind}{{\mathrm{ind}}}

\author{You-Cheng Chou}
\address{Institute of Mathematics, Academia Sinica,
Taipei, 10617, Taiwan
}
\email{bensonchou@gate.sinica.edu.tw}
\begin{document}

\maketitle
\begin{abstract}
    We give an alternate proof of the integrality conjecture of genus zero Gopakumar-Vafa type invariants on semi-positive varieties using algebraic geometry. The main technique is to relate Gopakumar-Vafa type invariants to quantum $K$-invariants and to utilize the integrality of the latter.
\end{abstract}
\section{Introduction}

Let $X$ be a smooth complex projective variety. (Cohomological) Gromov–Witten invariants of $X$ (with primary field) are defined to be
\[
\begin{split}
    \GW_{g,\beta}(\gamma_1,\dots,\gamma_n) & := \langle \gamma_1,\dots,\gamma_n \rangle_{g,n,\beta}^{X,H}
    \\
    &= \int_{[\M_{g,n}(X,\beta)]^{\vir}} \prod_{i=1}^n \ev_i^*(\gamma_i) \in \mathbb{Q},
\end{split}
\]
where $[\M_{g,n}(X,\beta)]^{\vir}$ is the virtual fundamental class, $\gamma_1,\dots,\gamma_n\in H^*(X)$, and $\ev_i$ is the evaluation map at the $i$-th marked point.


Due to the multiple cover contribution, Gromov-Witten invariants are in general not integers. For Calabi-Yau threefolds, the genus zero multiple cover formula was conjectured by physicists \cite{Candelas_Ossa_Green_Parkes}, reinterpreted in mathematical terms by Aspinwall-Morrison \cite{Aspinwall_Morrison}, and proved rigorously by Voisin \cite{Voisin} and many others. The genus one case was computed in physics \cite{Bershadsky_Cecotti_Ooguri_Vafa} and mathematics \cite{Graber_Pandharipande}. The higher genus case was proved by Faber-Pandharipande \cite{Faber_Pandharipande}. 

The multiple cover formulas lead to integral contributions to the ``BPS invariants'', sometimes in a subtle way. See \cite{Bryan} for some discussion. The BPS invariants, also called Gopakumar-Vafa invariants, can be defined via Gromov-Witten invariants via the remarkable formula (\ref{eqn:GVGW}) by Gopakumar-Vafa \cite{Gopakumar_Vafa_1998}. The integrality of this ad hoc definition has been shown by E. Ionel and
T. Parker \cite{Ionel_Parker_2018}.
\begin{equation} \label{eqn:GVGW}
\begin{split}
& \sum _{g=0}^{\infty }~\sum _{k=1}^{\infty }~\sum _{\beta \in H_{2}(M,\mathbb {Z} )_{>0}}{\GV}_{g,\beta}{\frac {1}{k}}\left(2\sin \left({\frac {k\lambda }{2}}\right)\right)^{2g-2}q^{k\beta }\\
&:=
\sum _{g=0}^{\infty }~\sum _{\beta \in H_{2}(M,\mathbb {Z} )_{>0}}{\GW}_{g,\beta} q^{\beta }\lambda ^{2g-2}.
\end{split}
\end{equation}

In \cite{Klemm_Pandharipande_2008},
A.~Klemm and R.~Pandharipande predicted the genus zero multiple cover formula to all Calabi-Yau $m$-folds for $m\geq 4$. They defined Gopakumar–Vafa type invariants (in terms of Gromov-Witten invariants) and conjectured their integrality. In \cite{Ionel_Parker_2018}, the definition has been generalized to semi-positive (including Calabi-Yau and Fano) variety $X$ with $\dim_{\mathbb{C}} X \geq 3$. It consists of two parts. 
\begin{itemize}
    \item If $-\beta \cdot K_X =0$, it is the multiple cover formula in the form given by Klemm-Pandharipande \cite{Klemm_Pandharipande_2008}
    \[
    \sum_{\substack{ \beta\in H_2(X,\mathbb{Z})_{>0} \\ \beta \cdot K_X=0 }} \GV_{0,\beta}(\gamma_1,\dots,\gamma_n)\sum_{r=1}^{\infty} \frac{q^{r\beta}}{r^{3-n}} := 
    \sum_{\substack{ \beta\in H_2(X,\mathbb{Z})_{>0} \\ \beta \cdot K_X=0 }} \GW_{0,\beta}(\gamma_1,\dots,\gamma_n)q^{\beta}.
    \]
    \item If $-\beta \cdot K_X > 0$, 
    \[
    \GV_{0,\beta} (\gamma_1,\dots,\gamma_n) := \GW_{0,\beta}(\gamma_1,\dots,\gamma_n).
    \]
\end{itemize}
In this paper, we write $\GW_{\beta}(\cdots)$ and $\GV_{\beta}(\cdots)$ for simplicity since we only consider the genus zero case.

The integrality of genus zero Gopakumar-Vafa type invariants has been proved by Ionel–Parker using symplectic geometry.
\begin{theorem}[{\cite[Theorem~9.2]{Ionel_Parker_2018}}] \label{Theorem}
Let $X$ be a semi-positive variety and $\gamma_i \in H^{2*}(X;\mathbb{Z})$ be cohomology classes with even degree. The invariants $\GV_{\beta}(\gamma_1,\dots,\gamma_n)$ are integers. 
\end{theorem}
The goal of this paper is to give an alternate proof of this result using only algebraic geometry. In section \ref{sec:quantumK}, we recall some necessary ingredients on A.~Givental and his collaborators' framework about quantum $K$-theory (including the virtual Kawasaki's Hirzebruch-Riemann-Roch formula, the fake theory, and the stem theory). 
In section \ref{sec:proof}, we relate Gopakumar-Vafa type invariants to quantum $K$-invariants. The semi-positive assumption ensures that the virtual dimension of twisted sector (Kawasaki strata) is less than or equal to the virtual dimension of the untwisted one. It will greatly simplify the computation. Finally, we show that the integrality of quantum $K$-invariants will imply the integrality of Gopakumar-Vafa type invariants.

\subsection*{Acknowledgement} I wish to thank Chin-Lung Wang, Sz-Sheng Wang, Nawaz Sultani, and Wille Liu for discussions about this work. Special thanks to Yuan-Pin Lee for reading the early draft of this paper and giving useful comments. This research is supported by Academia Sinica.

\section{Quantum $K$-theory} \label{sec:quantumK}
In this section, we give a limited introduction of Givental and his collaborators' framework about quantum $K$-theory.

\emph{$K$-theoretic Gromov-Witten invariants} are \emph{integral} invariants \cite{Givental_2000, Lee_2004}.
For any smooth projective variety $X$, they are defined as
\[
\begin{split}
\QK_{g,\beta} (\Gamma_1,\dots,\Gamma_n) &:= 
\langle \Gamma_1,\dots,\Gamma_n \rangle_{g,n,\beta}^{X,K}
\\
&=\chi \Big(  \M_{g,n}(X,\beta) ; \Big(  \otimes_{i=1}^n \ev_i^*(\Gamma_i)  \Big) \otimes \O^{\vir}  \Big) \in \mathbb{Z}.
\end{split}
\]
Here $\Gamma_1,\dots,\Gamma_n \in K^0(X)$, the topological $K$-theory of $X$, and $\O^{\vir}$ is the virtual structure sheaf on $\M_{g,n}(X,\beta)$. We denote $\QK_{\beta}(\dots)$ for genus 0 case. It is useful to encode genus zero invariants into power series, called (big) $J$-function. We recall its construction below.


\subsection{The symplectic loop space formalism}

Let $\Lambda = \mathbb{Q}[\![Q]\!]$ be the Novikov ring and 
\[
 \mathbf{K} := K^0(X)\otimes \Lambda.
\]
Givental's loop space for quantum $K$-theory is defined as
\[
\K := K^0(X)(q) \otimes \Lambda.
\]
$\K$ has a natural symplectic structure with the symplectic form $\Omega$,
\[
\K \ni f,g \mapsto \Omega(f,g) := \Big( \Res_{q=0} + \Res_{q=\infty} \Big) ( f(q), g(q^{-1}) )^{K} \frac{dq}{q}.
\]
Here $( \cdot, \cdot )^{K}$ denotes the $K$-theoretic intersection pairing on $\mathbf{K}$:
\[
(a,b)^{K} := \chi(X, a \otimes b) = \int_X \td (T_X) \ch(a) \ch(b).
\]
$\K$ admits the following Lagrangian polarization with respect to $\Omega$:
\[
\begin{split}
    \K &= \K_+ \oplus \K_-
    \\
    & := \mathbf{K}[q,q^{-1}] \, \oplus \, \{ f(q) \in \K | f(0) \neq \infty ,\ f(\infty) =0 \}.
\end{split}
\]
\begin{definition}\label{def:KbigJ}
The \emph{big $J$-function} of $X$ in the quantum $K$-theory is defined as a map $\K_+ \rightarrow \K$:
\[
   \mathbf{t} \mapsto J^{\KK}(\mathbf{t}) := (1-q) + \mathbf{t}(q) + \sum_{\alpha} \Phi^{\alpha}\sum_{n,\beta} \frac{Q^{\beta}}{n!} \langle \frac{\Phi_{\alpha}}{1-qL}, \mathbf{t}(L),\dots, \mathbf{t}(L) \rangle ^{X,K}_{0,n+1,\beta},
\]
where $\{ \Phi_{\alpha} \}$ and $\{\Phi^{\alpha}\}$ are Poincar\'e-dual basis of $K^0(X)$ with respect to $( \cdot, \cdot )^{K}$. 

\end{definition}

\subsection{$J^K$-function as a graph sum via Kawasaki's HRR}
Let $I\mathcal{M} = \sqcup_i \mathcal{M}_i$ be the inertia stack of $\mathcal{M}$, with $\mathcal{M}_i$ connected components. Following Givental, we refer to them as \emph{Kawasaki strata}. Kawasaki's formula \cite{Kawasaki_1979} reads
\[
\chi(\mathcal{M}, E) = \sum_{i} \frac{1}{m_i} \int_{\mathcal{M}_i} \td(T_{\mathcal{M}_i}) \ch \Big(  \frac{ \Tr (E|_{\mathcal{M}_i}) }{ \Tr ( \Lambda^* N_{\mathcal{M}_i}^* ) } \Big),
\]
where the sum over $i$ runs through all connected components.

Applying the virtual Kawasaki's Hirzebruch--Riemann--Roch formula for Deligne--Mumford stacks (VKHRR) \cite{Tonita_VKHRR} on $\M_{0,n}(X,\beta)$, we get  
\[  \begin{split}
J^{\KK}(\mathbf{t}) &=(1-q) + \mathbf{t}(q) + \sum_{n,\beta,\alpha} \frac{Q^{\beta} \Phi^{\alpha}}{n!}\langle \frac{\Phi_{\alpha}}{1-qL}, \mathbf{t}(L),\dots, \mathbf{t}(L)  \rangle^{X,K}_{0,n+1,\beta}
\\
&= (1-q) + \mathbf{t}(q) + \sum_{n,\beta,\alpha} \frac{Q^{\beta} \Phi^{\alpha}}{n!}\sum_{\zeta} \langle \frac{\Phi_{\alpha}}{1-qL}, \mathbf{t}(L),\dots,\mathbf{t}(L)  \rangle^{X_{\zeta}}_{0,n+1,\beta},
\end{split}\]
where in $\sum_{\zeta}$ $\zeta$ runs through all roots of unity (including 1), $X_{\zeta}$ stand for the collection of inertia stacks (``Kawasaki strata'') where $g$ acts on $L_1$ with eigenvalue $\zeta$, and $\langle ... \rangle^{X_{\zeta}}$ denote the contributions of $X_{\zeta}$ in the Riemann--Roch formula. In other words, $\langle ...\rangle^{X,K}$ represent (true) quantum $K$-invariants while $\langle ... \rangle^{X_{\zeta}}$ stand for (collections of) \emph{cohomological} invariants. Symbolically, we have 
\[
\langle \cdots \rangle^{X,K} = \sum_{\zeta} \langle \cdots\rangle^{X_{\zeta}}.
\]
Let $\zeta$ be a primitive $r$-th roots of unity. Define
\[
\begin{split}
    {\rm arm}(q) &:= \sum_{n,\beta \neq 0,\alpha} \frac{Q^{\beta} \Phi^{\alpha}}{n!}\sum_{\zeta' \neq 1} \langle \frac{\Phi_{\alpha}}{1-qL}, \mathbf{t}(L),\dots,\mathbf{t}(L)  \rangle^{X_{\zeta'}}_{0,n+1,\beta}\in \mathbf{K}[\![ 1-q ]\!]; 
    \\
    {\rm leg}_{r}(q) &:= \Psi^r ({\rm arm}(q)|_{\mathbf{t}=0} )\in \mathbf{K}[\![ 1-q ]\!];
    \\
    {\rm tail}_{\zeta}(q) &:= \sum_{n,\beta \neq 0,\alpha} \frac{Q^{\beta} \Phi^{\alpha}}{n!}\sum_{\zeta' \neq \zeta} \langle \frac{\Phi_{\alpha}}{1-qL}, \mathbf{t}(L),\dots,\mathbf{t}(L)  \rangle^{X_{\zeta'}}_{0,n+1,\beta}\in \mathbf{K}[\![ 1- \zeta q ]\!],
\end{split}
\]
where $\zeta'$ in the above sums are arbitrary roots of unity and
$\Psi^r$ are the \emph{Adams operations}. Recall that Adams operations are additive and multiplicative endomorphisms of $K$-theory or more generally $\lambda$-rings, acting on line bundles by $\Psi^r(L) = L^{\otimes r}$. Here $\Psi^r$ also act on the Novikov variables by $\Psi^r(Q^{\beta}) = Q^{r\beta}$ and on $q$ by $\Psi^r(q) = q^r$. 

We recall the following two propositions which relate contributions from Kawasaki strata into fake $K$-theory and stem theory respectively.

The fake quantum $K$-theory, defined by ``naively'' applying the virtual Riemann--Roch for schemes to stacks
\[
\langle \tau_{d_1}(\Gamma_1)\dots \tau_{d_n}(\Gamma_n) \rangle^{X, \fake}_{g,n,\beta} := \int_{[\M_{g,n}(X,\beta)]^{\vir} } \td(T_{\M_{g,n}(X,\beta)}^{\vir}) \ch ( \otimes_{i=1}^n \ev_i^*(\Gamma_i) L_i^{d_i} ),
\]
where $[\M_{g,n}(X,\beta)]^{\vir}$ is the (cohomological) virtual fundamental class and $T^{\vir}_{\M_{g,n}(X,\beta)}$ is the virtual tangent bundle.
\begin{proposition}[{\cite[Proposition~1]{Givental_Tonita_2011}}] \label{prop_adelic_fake}
\[
    J^{K}(\mathbf{t})|_{q=1} = J^{\fake}(\mathbf{t} + {\rm arm}),
\]
where $()|_{q=1}$ is the Laurent series expansion at $q=1$ (of the rational function).
\end{proposition}

For $\zeta\neq 1$, a primitive $r$-th roots of unity, we consider the \emph{stem space}, which is isomorphic to
\[
 \M_{0,n+2,\beta}^X(\zeta) := \M_{0,n+2} \left( [X/\mathbb{Z}_r], \beta; (g,1,\dots,1,g^{-1}) \right).
\]
Here the group elements $g, 1, g^{-1}$ signal the twisted sectors in which the marked points lie. The notation $[...]^{X_{\zeta}}$ will be reserved for stem contributions
\[ 
\begin{split}
    &\Big[ T_1(L), T(L),\dots, T(L), T_{n+2}(L) \Big]^{X_{\zeta}}_{0,n+2,\beta}
    \\
 := &\int_{[\M_{0,n+2,\beta}^X(\zeta)]^{\vir}} \td( T_{\M} ) \ch \left( 
    \frac{ \ev_1^* (T_1(L)) \ev_{n+2}^* (T_{n+2}(L)) \prod_{i=2}^{n+1} \ev_i^* T(L) }
    {\Tr (\Lambda^* N^*_{\M})}  \right),
\end{split}
\]
where $[\M_{0,n+2,\beta}^X(\zeta)]^{\vir}$ is the virtual fundamental class, $T_{\M}$ is the (virtual) tangent bundle to $\M_{0,n+2,\beta}^X(\zeta)$, and $N_{\M}$ is the (virtual) normal bundle of $\M_{0,n+2,\beta}^X(\zeta) \subset \M_{0,nr+2}(X,r\beta)$. 

\begin{proposition}[{ \cite[\S~8]{Givental_Tonita_2011} }] \label{prop_adelic_stem}
Let $\zeta$ be a primitive $r$-th roots of unity.
We have
\[
\begin{split}
    & \sum_{n,\beta,\alpha}  \frac{Q^{\beta} \Phi^{\alpha}}{n!}\langle \frac{\Phi_{\alpha}}{1-qL}, \mathbf{t}(L),\dots,\mathbf{t}(L)  \rangle^{X_{\zeta}}_{0,n+1,\beta} 
    \\
   = & \sum_{n,\beta,\alpha} \frac{Q^{r\beta}\Phi^{\alpha}}{n!} \Big[ \frac{\Phi_{\alpha}}{1-q\zeta L^{1/r}}, {\rm leg}_{r}(L),\dots, {\rm leg}_{r}(L),  \delta_{\zeta}(L^{1/r}) \Big]^{X_{\zeta}}_{0,n+2,\beta},
\end{split}
\]
where
\[
\delta_{\zeta}(q) = (1- \zeta^{-1} q) + t(\zeta^{-1} q) + {\rm tail}_{\zeta}(\zeta^{-1} q).
\]
\end{proposition}

\section{Proof of Theorem~\ref{Theorem}} \label{sec:proof}
In this section, we assume $X$ to be a semi-positive variety of dimension $m$. In section~\ref{Sec:3.1}, we recall basic facts about (twisted) Gromov-Witten invariants on $X$. In section \ref{sec:3.2}, we prove three formulas relating Gopakumar-Vafa type invariants to quantum $K$-invariants. Finally, in the last section, we show that the integrality of quantum $K$-invariants implies the integrality of Gopakumar-Vafa type invariants.
\subsection{Twisted Gromov-Witten theory} \label{Sec:3.1}
Let $\zeta$ be a primitive $r$-th roots of unity. In the stem theory, we consider the twisting class
\[
\td (T_{\M}) \ch \Big( \frac{1}{\Tr (\Lambda^* N^*_{\M})} \Big).
\]
Such twisting classes come from the deformation theory of the moduli of stable maps, $\M := \M^X_{0,n+2,\beta}(\zeta)$, and consist of three parts. See e.g., \cite[Section~8]{Givental_Tonita_2011}.
\begin{enumerate}
    \item[type $A$.] $\displaystyle \td(\pi_*^K \ev^*(TX))\prod_{k=1}^{r-1} \td_{\zeta^k} (\pi_*^K \ev^*(T_X \otimes \mathbb{C}_{\zeta^k})) $, where $\pi: \mathcal{C} \rightarrow \M$ and $\ev: \mathcal{C} \rightarrow X/ \mathbb{Z}_r$ form the universal (orbifold) stable map diagram. $\mathbb{C}_{\zeta^k}$ is the line bundle over $B\mathbb{Z}_r$ with $g$ acts by $\zeta^k$, and for any line bundle $l$, define the invertible multiplicative characteristic classes
    \[
    \td(l) := \frac{c_1(l)}{1-e^{-c_1(l)}}, \quad \td_{\lambda}(l) := \frac{1}{1-\lambda e^{-c_1(l) }}.
    \]
    \item[type $B$.] $\displaystyle \td(\pi_*^K(-L^{-1})) \prod_{k=1}^{r-1} \td_{\zeta^k} (\pi_*^K(-L^{-1} \otimes \ev^* (\mathbb{C}_{\zeta^k}) ))$, where $L=L_{n+3}$ is the universal cotangent line bundle of 
    \[
    \mathcal{C} \cong \M_{0,n+3}^{X/\mathbb{Z}_r ,\beta} (g, 1,\dots,1,g^{-1},1).
    \]
    \item[type $C$.] $\displaystyle \td^{\vee}(-\pi_*^K i_* \mathcal{O}_{Z_g}) \td^{\vee}(-\pi_*^K i_*\mathcal{O}_{Z_1}) \prod_{i=1}^{k-1}\td^{\vee}_{\zeta^k}(-\pi_*^K i_*\mathcal{O}_{Z_1})$, where $Z_1$ stands for unramified nodal locus, and $Z_g$ stands for ramified one with $i: Z \rightarrow \mathcal{C}$ the embedding of nodal locus. For any line bundle $l$,
    \[
     \td^{\vee}(l) = \frac{-c_1(l)}{1-e^{c_1(l)}}, \quad \td^{\vee}_{\lambda}(l) = \frac{1}{1-\lambda e^{c_1(l) }}.
    \]
\end{enumerate}
We have the following observations.
\begin{lemma} \label{lemma_rpower}
\[
\td (T_{\M}) \ch \Big( \frac{1}{\Tr (\Lambda^* N^*_{\M})} \Big) =: r^{-(n - K_X\cdot\beta)} + T_{0,n+2,\beta}(\zeta),
\]
where $T_{0,n+2,\beta}(\zeta)\in H^{>0}(\M^X_{0,n+2,\beta}(\zeta))$, with $0$, $n$, $\beta$, and $\zeta$ inherited from $\M = \M^X_{0,n+2,\beta}(\zeta)$.
\end{lemma}
\begin{proof}
Note that $N_{\M}$ has virtual dimension $(r - 1)(n - K_X\cdot\beta)$ since $\M$ is considered as Kawasaki strata in $\M_{0,nr+2}(X, r\beta)$. 

Only the twisting classes on normal bundle give nontrivial constant. It is of the following form:
\[
\prod_{k=1}^{r-1} \td_{\zeta^k} \Big( \pi^K_*(  E \otimes \ev^* (\mathbb{C}_{\zeta^k}) )\Big),
\]
where $E$ is a virtual bundle of rank $n-K_X\cdot\beta$. The constant term is given by
\[
\Big( \prod_{k=1}^{r-1} \frac{1}{1-\zeta^k} \Big)^{n-K_X\cdot \beta} = r^{-(n-K_X\cdot\beta)}.
\]

\end{proof}

\begin{lemma} \label{lemma_vanishing}
For $\beta\in H_2(X,\mathbb{Z})_{>0}$, assume $\deg_{\mathbb{C}}\gamma_1 \geq m-K_X\cdot \beta -1$ . We have
\[
\langle \tau_{k_1}(\gamma_1),\dots, \tau_{k_n}(\gamma_n) \rangle_{g,n,\beta}^{\#} =0,
\]
where $\#$ denotes cohomological GW invariants with any twisting.

In $K$-theory, assume $\deg_{\mathbb{C}} \ch (\Gamma_1) \geq m - K_X\cdot \beta -1$ . We have
\[
    \langle \tau_{k_1}(\Gamma_1),\dots, \tau_{k_n}(\Gamma_n) \rangle_{g,n,\beta}^K =0.
\]
\end{lemma}
\begin{proof}
For the cohomological case, let
\[
 \pi_1 : \M_{g,n}(X, \beta) \to \M_{g,1} (X, \beta)
\]
be the forgetful map forgetting the last $n-1$ marked points and $T\in H^*(\M_{g,n}(X,\beta))$ be the twisting class.
By projection formula
\[
\begin{split}
\int_{[\M_{g,n}(X, \beta)]^{\vir}} T \ &\prod_{i=1}^{n} \Big(\psi_i^{k_i} \ev_i^*\gamma_i\Big) 
\\
&= \int_{[\M_{g,1}(X, \beta)]^{\vir}} (\ev_1^* \gamma_1 ) \ (\pi_1)_*\left( T \ \psi_1^{k_1}\prod_{i=2}^{n} \psi_i^{k_i} \ev_i^*\gamma_i \right) .
\end{split}
\] 
The last equation vanishes for dimension reason. The same argument works in $K$-theory.
\end{proof}
From Lemma~\ref{lemma_vanishing} and the definition of $\arm$, $\leg$, and $\tail$, we have
\begin{corollary} \label{corollary_arm_leg_tail}
\[
\begin{split}
    & \ch (\arm (q)) \in H^{\geq 2(2 + K_X\cdot\beta)}(X)[\![ 1-q,Q ]\!];
    \\
    & \ch (\leg_r (q)) \in H^{\geq 2(2 + K_X\cdot\beta)}(X)[\![ 1-q,Q ]\!];
    \\
    & \ch (\tail_{\zeta}(q)) \in H^{\geq 2(2 + K_X\cdot\beta)}(X)[\![ 1-\zeta q,Q ]\!].
\end{split}
\]
\end{corollary}

\subsection{Relations between GV and QK}\label{sec:3.2}
In this subsection, we prove three formulas relating Gopakumar-Vafa type invariants to quantum $K$-invariants.
\begin{proposition} \label{prop_QK=GV_1pt} Let $\gamma \in H^{2(m-K_X\cdot \beta-2)}(X)$ and $\Gamma:= \ch^{-1} (\gamma) \in K^0(X)_{\mathbb{Q}}$. We have
\[
    \GV_{\beta}(\gamma) = \QK_{\beta}(\Gamma) + \delta_{K_X\cdot \beta,0} \sum_{\substack{r|\ind(\beta) \\ r\neq 1}}  \mu(r) r \ \QK_{\beta/r} (\Gamma),
\]
where $\ind(\beta):=\max\{ k\in \mathbb{N} | \beta/k \in H_2(X,\mathbb{Z}) \}$ and $\mu$ is the M\"obius function.
\end{proposition}
\begin{proof}
We compute $\QK_{\beta}(\Gamma)$ using VKHRR and divide the graph sum of Kawasaki strata into three types, type 1, type 2, and type 3. 

 \begin{figure}[ht]
    \centering
    \begin{tikzpicture}[scale=2.6]
    \draw (0,0) .. controls (0.4,0.5) and (0.8,-0.5) .. (1.2,0);
    \filldraw[thick,-to](0,0)--(0,0.4);
    \filldraw[black] (-0.2,0.2) node[anchor=west]{\tiny $\Gamma$};

    \filldraw[ultra thick] (2,0) -- (3,0);
    \filldraw[thick,-to](2,0)--(2,0.4);
    \filldraw[black] (1.8,0.2) node[anchor=west]{\tiny $\Gamma$};
    \filldraw[black] (2.3,0.1) node[anchor=west]{\tiny ${\rm stem}$};
    \filldraw[black] (2.9,0.1) node[anchor=west]{\tiny $\delta_{\zeta}$};
    \filldraw[black] (3,0) circle(0.7 pt);
    
    \filldraw[ultra thick] (4,-0.2) -- (4,0.5);
    \filldraw[thick, -to] (4,0.5) -- (3.6,0.5);
    \filldraw[black] (3.7,0.6) node[anchor=west]{\tiny $\Gamma$};
    \draw (4, 0.3) .. controls (4.25, 0.4) and (4.5, 0.2) .. (4.7, 0.3);
    \filldraw[black] (4,-0.2) node[anchor=west]{\tiny $\delta_{\zeta}$};
    \filldraw[black] (4.4,0.35) node[anchor=west]{\tiny $\leg$};
    \filldraw[black] (4,-0.2) circle(0.7 pt);
    \end{tikzpicture}
    \caption{type 1 (left), type 2 (middle), type 3 (right)}
    \label{}
    \end{figure}
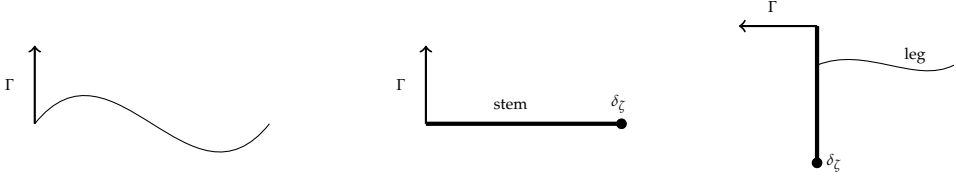

We explain the above graph as follows. See \cite{Chou_Lee_2022_QKGVI} or \cite{Givental_Tonita_2011} for more detail.
\begin{itemize}
    \item The arrow with $\Gamma$ denotes the first marked point.
    \item The stem curve is emphasized as thick line. 
    \item The vertical line denotes the curve that contracts to a point
    under the stable map.
    \item Type 1 graph denotes the Kawasaki strata of the fake theory. Type 2 graph denotes the Kawasaki strata of the stem theory with the stem curve does not contract. Type 3 graph denotes the Kawasaki strata of the stem theory with the stem curve contracts.
\end{itemize}

The type 1 contribution gives $\GW_{\beta}(\gamma)$ since $\gamma$ is of the top degree. The twisting class $\td(T_{\M})$ gives no contribution in the fake theory computation.

The type 2 contribution has two cases. If $-\beta\cdot K_X=0$, we have
\[
\begin{split}
    &\sum_{\substack{r|\ind(\beta) \\ r\neq 1}} \sum_{\zeta^r = 1} \Big[ \Gamma, \delta_{\zeta}(L^{1/r})  \Big]^{X_{\zeta}}_{0,2,\beta / r}
    \\
    &=\sum_{\substack{r|\ind(\beta) \\ r\neq 1}} \sum_{\zeta^r=1}\Bigg( -\zeta^{-1} \Big\langle \gamma, \frac{\psi}{r} \Big\rangle^{X/\mathbb{Z}_r,H}_{0,2,\beta/r} 
    \\
    & \hspace{35mm} + \frac{1}{r} \Big\langle \Gamma, \Psi^r \Big( \sum_{\alpha} \Phi^{\alpha} \langle \Phi_{\alpha} \rangle^{X,\fake}_{0,1,\beta/r} \Big), 1-\zeta^{-1}\Big\rangle^{X/\mathbb{Z}_r,K}_{0,3,0} \Bigg)
    \\
    &= \sum_{r|\ind(\beta)} \Bigg( \frac{\mu(r)}{r^2} \GW_{\beta /r} (\gamma) + \Big(\phi(r)-\mu(r) \Big) \GW_{\beta/r}(\gamma)  \Bigg),
\end{split}
\]
where $\sum_{\zeta^r=1}$ denotes the sum over primitive $r$-th roots of unity.
The first term after the first equality comes from the untwisted contribution. The second term comes from the twisting of type $C$ (type $A$ and $B$ gives no contribution).
In the last equality, we use the dilaton equation and the following formulas:
\[
\sum_{\zeta^r=1} \zeta = \sum_{\zeta^r=1} \zeta^{-1} = \mu(r), \qquad  \sum_{\zeta^r=1} 1 = \phi(r).
\]
If $-\beta\cdot K_X > 0$, there is no contribution for type 2 for dimensional reason. Note that for $r \geq 2$, we have 
$$\vdim \M_{0,n}(X, \beta/r) < \vdim \M_{0,n}(X,\beta).$$

For the same reason, only $-\beta \cdot K_X=0$ case will give nonzero contribution for type 3. For $r\geq 2$, we have
\[
\begin{split}
& \Coeff\Bigg( Q^{\beta} ; \sum_{\substack{r|\ind(\beta) \\ r\neq 1}} \sum_{\zeta^r = 1} \Big[ \Gamma , \leg_r(L)(q=1), \delta_{\zeta}(L^{1/r}) \Big]_{0,3,0}^{X,K} \Bigg) \\
& = \sum_{\substack{r|\ind(\beta) \\ r\neq 1}} \sum_{\zeta^r = 1} \frac{1}{r} \Big\langle \Gamma, \Coeff \Big( Q^\beta ; \leg_r(L)(q=1) \Big), 1- \zeta^{-1} \Big\rangle_{0,3,0}^{X/\mathbb{Z}_r, K}
\\
& = \Big(\phi(r)-\mu(r)\Big) \Big( \sum_{k| \frac{\ind{\beta}}{r} } \GV_{\beta/(rk)}(\gamma) - \GW_{\beta/r}(\gamma) \Big),
\end{split}
\]
We only consider the constant part of $\delta_{\zeta}$ for dimensional reason since $\ch(\leg(L))(q=1) \in H^{\geq 4}[\![Q]\!]$ by Corollary~\ref{corollary_arm_leg_tail}. 
The constant $r^{-1}$ comes from the twisting class, see Lemma~\ref{lemma_rpower}. In the second equality, we use the formula of $\leg_r(L)(q=1)$ discussed in Remark \ref{rmk}.

Combine all three contributions, we have
\[
\begin{split}
    \QK_{\beta}(\Gamma) &= \GW_{\beta}(\gamma) 
    \\
    &+ \delta_{K_X\cdot \beta, 0} \sum_{\substack{r|\ind(\beta) \\ r\neq 1}}\Bigg( \frac{\mu(r)}{r^2}  \GW_{\beta/r}(\gamma) + \Big(\phi(r)-\mu(r) \Big) \sum_{k| \frac{\ind \beta}{r}} \GV_{\beta/(rk)}(\gamma) \Bigg)
    \\ &= \GV_{\beta}(\gamma) + \delta_{K_X \cdot \beta,0} \sum_{\substack{ r| \ind(\beta) \\ r\neq 1}} r\ \GV_{\beta/r} (\gamma).
\end{split}
\]
For the second equality, we use the identities:
\[
    \sum_{k|r} \phi(k) = r, \qquad \sum_{k|r} \mu (k) = 0.
\]
The proposition follows from the M\"obius inverse formula.
\end{proof}

\begin{remark} \label{rmk} 
With the same notation as in Proposition~\ref{prop_QK=GV_1pt}, let $\check{\gamma}$ be the Poincar\'e dual of $\gamma$. We have
\[
\begin{split}
    \Coeff \Big( \check{\gamma} ; \ch(\leg_r (L))(q=1)  \Big) 
    & = r^2 \sum_{\beta} \Big( \sum_{k | \frac{\ind(\beta)}{r}}\GV_{\beta/(rk)} (\gamma) - \GW_{\beta/r}(\gamma) \Big)Q^{\beta}.
\end{split}
\]
It can be computed by induction on the Novikov variable $Q^{\beta}$. It is quite involved. See \cite[Section~4]{Chou_Lee_2022_QKGVI} for an explicit computation and \cite[Section~3]{Chou_Lee_2022_QKGVII} for the formula on Calabi-Yau threefold. For reader's convenience, we give a comparison of the above formula with the one in \cite{Chou_Lee_2022_QKGVII}. Let $\phi_j$ be a divisor. The corresponding term in \cite{Chou_Lee_2022_QKGVII} gives
\[
\begin{split}
    & r^2\sum_{\ind \beta =1} \sum_{d=1}^{\infty} (\beta \cdot \phi_j) \Big( \GV_{d\beta}^{(1)} - \frac{\GV_{d\beta}^{(3)}}{d^2} \Big) Q^{rd\beta}
    \\
    & := r^2 \sum_{\ind \beta =1} \sum_{d=1}^{\infty} (\beta \cdot \phi_j) \sum_{k|d} \Big( \frac{d}{k} \GV_{d\beta/k} - \frac{d\GV_{d\beta/k}}{k^3} \Big) Q^{rd\beta}
    \\
    & = r^2 \sum_{\ind \beta =1} \sum_{d=1}^{\infty} \Big( \sum_{k|d} \GV_{d\beta/k}(\phi_j) - \GW_{d\beta}(\phi_j) \Big) Q^{rd\beta}.
\end{split}
\]
For the first equality, we use the definition $\GV_{d\beta}^{(\gamma)} := \sum_{k|d} (\frac{d}{k})^{\gamma} \GV_{d\beta/k}$. The second equality follows from the divisor equation and the relation between $\GW$ and $\GV$. One can see that the two formulas are equivalent.
\end{remark}

\begin{proposition} \label{prop_QK=GV_2pt} Let $\gamma_1 \in H^{2j}(X)$ and $\gamma_2 \in H^{2(m-K_X\cdot\beta -1-j)}(X)$. Let $\Gamma_i := \ch^{-1}(\gamma_i) \in K^0(X)_{\mathbb{Q}}$ for $i=1$ and $2$. We have 
\[
    \GV_{\beta}(\gamma_1,\gamma_{2}) = \QK_{\beta}(\Gamma_1,\Gamma_2) +  \delta_{K_X\cdot \beta,0} \sum_{\substack{r|\ind(\beta) \\ r\neq 1}} \mu(r) \QK_{\beta/r}(\Gamma_1,\Gamma_{2}).
\]
\end{proposition}
\begin{proof}
Using the same approach as Propositon~\ref{prop_QK=GV_1pt}, we compute $\QK_{\beta}(\Gamma_1,\Gamma_2)$ from the type 1, type 2, and type 3 contributions in VKHRR.

The type 1 contribution gives $\GW_{\beta}(\gamma_1,\gamma_{2})$ since there is no twisting contribution for dimensional reason.

The type 2 contribution has two cases. If $-K_X\cdot \beta=0$, we have 
\[
    \sum_{\substack{r|\ind(\beta) \\ r\neq 1}} \sum_{\zeta^r = 1} \Big[ \Gamma_1, \Gamma_2 \Big]^{X_{\zeta}}_{0,2,\beta / r} = \sum_{\substack{r|\ind(\beta) \\ r\neq 1}} \frac{\phi(r)}{r} \GW_{\beta / r}(\gamma_1, \gamma_2).
\]
If $-K_X\cdot\beta >0$, we have
\[
    \sum_{\substack{r|\ind(\beta) \\ r\neq 1}} \sum_{\zeta^r = 1} \Big[ \Gamma_1, \Gamma_2 \Big]^{X_{\zeta}}_{0,2,\beta / r} =0
\]
for dimensional reason.

There is no type 3 contribution for dimensional reason. More precisely, we have
\[
\Big[ \Gamma_1, \leg_r (L)(q=1), \Gamma_{2} \Big]^{X,K}_{0,3,0}=0
\]
since $\ch (\leg (L))(q=1) \in H^{\geq 2(2+K_X\cdot\beta)}(X)[\![Q ]\!]$ by Corollary~\ref{corollary_arm_leg_tail}.

Combine all three contributions, we have
\[
\begin{split}
    \QK_{\beta}(\Gamma_1, \Gamma_2) &= \GW_{\beta}(\gamma_1, \gamma_2) + \delta_{K_X\cdot\beta,0} \sum_{\substack{r|\ind(\beta) \\ r\neq 1}} \frac{\phi(r)}{r} \GW_{\beta/r}(\gamma_1, \gamma_2)
    \\
    &=\GV_{\beta}(\gamma_1,\gamma_2) + \delta_{K_X\cdot\beta,0} \sum_{\substack{r|\ind(\beta) \\ r\neq 1}} \GV_{\beta/r}(\gamma_1,\gamma_2).
\end{split}
\]
In the last equality, we rewrite all $\GW$ in terms of $\GV$ and use the following identity:
\[
    \sum_{k|r} \frac{\phi(k)}{r} = 1.
\]
The proposition follows from the M\"obius inverse formula.
\end{proof}
\begin{proposition} \label{prop_QK=GV_npt} Let $\gamma_i \in H^*(X)$ for $i=1,\dots,n$ with $n\geq 3$ and let $\Gamma_i := \ch^{-1}(\gamma_i) \in K^0(X)_{\mathbb{Q}}$. We have
\[
    \GV_{\beta}(\gamma_1,\dots,\gamma_n) = \QK_{\beta}(\Gamma_1,\dots,\Gamma_n) +   \delta_{K_X\cdot \beta,0} \sum_{\substack{r|\ind(\beta) \\ r\neq 1}}\mu(r) r^{n-3}\QK_{\beta/r}(\Gamma_1,\dots,\Gamma_{n}).
\]
\end{proposition}
\begin{proof}
Using the same approach as Propositon~\ref{prop_QK=GV_1pt}, we compute $\QK_{\beta}(\Gamma_1,\dots,\Gamma_n)$ from the type 1, type 2, and type 3 contributions in VKHRR.

Only type 1 contribution is nonvanishing and we have
\[
\begin{split}
    \QK_{\beta}(\Gamma_1,\dots,\Gamma_n) &= \GW_{\beta}(\gamma_1,\dots,\gamma_n).
\end{split}
\]
The proposition follows from rewriting $\GW$ to $\GV$ and apply M\"obius inverse formula.
\end{proof}

\subsection{Conclusion of the proof} Let $X$ be a smooth projective variety. We claim that for any $\gamma \in H^{2i}(X;\mathbb{Z}) / {\rm torsion} \subset H^{2i}(X;\mathbb{Q})$. There exists $\Gamma \in K^0(X)$ such that $\ch(\Gamma) = \gamma$ (mod $ H^{>2i}(X;\mathbb{Q}) $). One can prove it by giving $X$ a CW complex structure and  induction on the number of cells on $X$. See the argument in \cite[Proposition~4.5]{Hatcher_VB_Ktheory}. 

Theorem~\ref{Theorem} follows from the formulas in Proposition~\ref{prop_QK=GV_1pt}, Proposition~\ref{prop_QK=GV_2pt}, and Proposition~\ref{prop_QK=GV_npt} with only one difference. Instead of defining $\Gamma_i$ as $\ch^{-1}(\gamma_i)$ which defines over $\mathbb{Q}$, we take $\Gamma_i$ from the preceding claim. Note that the higher degree ambiguity will not change the computation in Proposition~\ref{prop_QK=GV_1pt}, Proposition~\ref{prop_QK=GV_2pt}, and Proposition~\ref{prop_QK=GV_npt} for dimensional reason. The integrality of quantum $K$-invariants then implies the integrality of Gopakumar-Vafa type invariants. This completes the proof.

\bibliographystyle{alpha}
\bibliography{zbib}

\begin{thebibliography}{CdlOGP92}

\bibitem[AM93]{Aspinwall_Morrison}
Paul~S. Aspinwall and David~R. Morrison.
\newblock Topological field theory and rational curves.
\newblock {\em Comm. Math. Phys.}, 151(2):245--262, 1993.

\bibitem[BCOV93]{Bershadsky_Cecotti_Ooguri_Vafa}
M.~Bershadsky, S.~Cecotti, H.~Ooguri, and C.~Vafa.
\newblock Holomorphic anomalies in topological field theories.
\newblock {\em Nuclear Phys. B}, 405(2-3):279--304, 1993.

\bibitem[Bry01]{Bryan}
Jim Bryan.
\newblock Multiple cover formulas for {G}romov-{W}itten invariants and {BPS}
  states.
\newblock In {\em Proceedings of the {W}orkshop ``{A}lgebraic {G}eometry and
  {I}ntegrable {S}ystems related to {S}tring {T}heory'' ({K}yoto, 2000)},
  number 1232, pages 144--159, 2001.

\bibitem[CdlOGP92]{Candelas_Ossa_Green_Parkes}
Philip Candelas, Xenia~C. de~la Ossa, Paul~S. Green, and Linda Parkes.
\newblock A pair of {C}alabi-{Y}au manifolds as an exactly soluble
  superconformal theory.
\newblock In {\em Essays on mirror manifolds}, pages 31--95. Int. Press, Hong
  Kong, 1992.

\bibitem[CL22a]{Chou_Lee_2022_QKGVII}
Y.-C. Chou and Y.-P. Lee.
\newblock Quantum {$K$}-invariants and {G}opakumar-{V}afa invariants {II}.
  {C}alabi-{Yau} threefolds at genus zero.
\newblock {\em arXiv:2305.08480}, 2022.

\bibitem[CL22b]{Chou_Lee_2022_QKGVI}
You-Cheng Chou and Yuan-Pin Lee.
\newblock Quantum {$K$}-invariants and {G}opakumar-{V}afa invariants {I}. {T}he
  quintic threefold.
\newblock {\em arXiv:2211.00788}, 2022.

\bibitem[FP00]{Faber_Pandharipande}
C.~Faber and R.~Pandharipande.
\newblock Hodge integrals and {G}romov-{W}itten theory.
\newblock {\em Invent. Math.}, 139(1):173--199, 2000.

\bibitem[Giv00]{Givental_2000}
Alexander Givental.
\newblock On the {WDVV} equation in quantum {$K$}-theory.
\newblock volume~48, pages 295--304. 2000.
\newblock Dedicated to William Fulton on the occasion of his 60th birthday.

\bibitem[GP99]{Graber_Pandharipande}
T.~Graber and R.~Pandharipande.
\newblock Localization of virtual classes.
\newblock {\em Invent. Math.}, 135(2):487--518, 1999.

\bibitem[GT14]{Givental_Tonita_2011}
Alexander Givental and Valentin Tonita.
\newblock The {H}irzebruch-{R}iemann-{R}och theorem in true genus-0 quantum
  {K}-theory.
\newblock In {\em Symplectic, {P}oisson, and noncommutative geometry},
  volume~62 of {\em Math. Sci. Res. Inst. Publ.}, pages 43--91. Cambridge Univ.
  Press, New York, 2014.

\bibitem[GV98]{Gopakumar_Vafa_1998}
Rajesh Gopakumar and Cumrun Vafa.
\newblock Symplectic geometry of frobenius structures.
\newblock {\em arXiv:hep-th/9812127}, pages 1--19, 12 1998.

\bibitem[Hat03]{Hatcher_VB_Ktheory}
Allen Hatcher.
\newblock Vector bundles and k-theory.
\newblock {\em Im Internet unter http://www. math. cornell. edu/\~{} hatcher},
  2003.

\bibitem[IP18]{Ionel_Parker_2018}
Eleny-Nicoleta Ionel and Thomas~H. Parker.
\newblock The {G}opakumar-{V}afa formula for symplectic manifolds.
\newblock {\em Ann. of Math. (2)}, 187(1):1--64, 2018.

\bibitem[Kaw79]{Kawasaki_1979}
Tetsuro Kawasaki.
\newblock The {R}iemann-{R}och theorem for complex {$V$}-manifolds.
\newblock {\em Osaka Math. J.}, 16(1):151--159, 1979.

\bibitem[KP08]{Klemm_Pandharipande_2008}
A.~Klemm and R.~Pandharipande.
\newblock Enumerative geometry of {C}alabi-{Y}au 4-folds.
\newblock {\em Comm. Math. Phys.}, 281(3):621--653, 2008.

\bibitem[Lee04]{Lee_2004}
Y.-P. Lee.
\newblock Quantum {$K$}-theory. {I}. {F}oundations.
\newblock {\em Duke Math. J.}, 121(3):389--424, 2004.

\bibitem[Ton14]{Tonita_VKHRR}
Valentin Tonita.
\newblock A virtual {K}awasaki-{R}iemann-{R}och formula.
\newblock {\em Pacific J. Math.}, 268(1):249--255, 2014.

\bibitem[Voi96]{Voisin}
Claire Voisin.
\newblock A mathematical proof of a formula of {A}spinwall and {M}orrison.
\newblock {\em Compositio Math.}, 104(2):135--151, 1996.

\end{thebibliography}

\end{document}